\documentclass[reqno]{amsart}
\usepackage{amssymb}

\numberwithin{equation}{section}
\theoremstyle{plain}
\newtheorem{thm}{Theorem}
\newtheorem{prop}{Proposition}
\newtheorem*{kf}{Key Fact}
\theoremstyle{remark}

\newcommand{\Fbar}{{\overline{F}}}
\newcommand{\Z}{\mathbb{Z}}
\newcommand{\Q}{\mathbb{Q}}
\newcommand{\C}{\mathbb{C}}
\newcommand{\Ql}{{\mathbb{Q}_\ell}}

\newcommand{\Qp}{{\mathbb{Q}_p}}
\newcommand{\Fp}{{\mathbb{F}_p}}
\renewcommand{\O}{\mathcal{O}}
\newcommand{\m}{\mathfrak{m}}
\newcommand{\tensor}{\otimes}
\newcommand{\GL}{\mathrm{GL}}
\DeclareMathOperator{\gal}{Gal}
\DeclareMathOperator{\im}{Im}
\DeclareMathOperator{\codim}{codim}

\begin{document}
\title{Conductors of $\ell$-adic representations}
\author{Douglas Ulmer}
\address{School of Mathematics \\ 
Georgia Institute of Technology\\ 
Atlanta, GA 30332}
\email{ulmer@math.gatech.edu}

\date{\today}


\begin{abstract}
  We give a new formula for the Artin conductor of an $\ell$-adic
  representation of the Weil group of a local field of residue
  characteristic $p\neq\ell$.
\end{abstract}

\maketitle

\section{Introduction}
Our aim in this note is to give a new formula for the Artin conductor
of an $\ell$-adic representation of the Galois group of a
non-archimedean local field of residue characteristic $p\neq\ell$, see
Theorem~\ref{thm:main} in Section~\ref{s:main}.  The proof that this
formula is equivalent to the standard one is a simple unwinding of the
definitions, but the new formula has the virtue that it does not
require any auxiliary constructions such as semi-simplification or
Weil-Deligne representations, and it works equally well for
representations with finite or infinite inertial image.  We also take
the occasion to correct a typographical error in a canonical
reference. 

I would have imagined that formula \eqref{eq:int} is well known, and I
make no claim of priority, but a proof that it is equivalent to other
definitions does not seem to appear in in the
literature.\footnote{The formula appears as a definition in
  \cite[\S2.1]{DarmonDiamondTaylor97} and in the notes
  \cite{WieseGR}, but without proof that it is equivalent to the
  standard definition.}  We hope having a complete treatment will be
useful to the community.

Here is a brief explanation of the main issue: Let $F$ be a local
field with Weil group $W_F$.  The conductor of a representation $\rho$
of $W_F$ depends only on the restriction of $\rho$ to the inertia
group $I_F$, and it is defined in the first instance in
\cite{Artin31b} only for representations which factor through finite
quotients of $W_F$.  Since the image of an $\ell$-adic representation
restricted to inertia need not be finite, further discussion is
required.  This problem was solved by Serre, who in \cite{Serre70}
gave a definition of the conductor of an $\ell$-adic representation
using the key fact that there is a subgroup of inertia of finite index
on which the representation is unipotent.  This last fact was
conjectured by Serre and Tate and proven by Grothendieck, see
\cite[Appendix]{SerreTate68}.  Later, Deligne gave a definition using
what is now known as a Weil-Deligne representation, a technical device
with several uses in the Langlands program, see \cite{Deligne73} and
\cite{Tate79}.

For the convenience of the reader, we discuss the previous definitions
of the conductor (Artin's definition in Sections~\ref{s:Galois} through
\ref{s:ArtinConductor}, Serre's in Sections~\ref{s:l-adic} and
\ref{s:SerreConductor}, and Deligne's in Sections~\ref{s:WD} and
\ref{s:l-WD}) with full definitions, but minimal proofs.  In
Section~\ref{s:main} we give the new formula, and in
Section~\ref{s:twist} we give two applications, one of which motivated
this work.

We thank Romyar Sharifi and Emmanuel Kowalski for pointing out
\cite{WieseGR} and \cite{DarmonDiamondTaylor97} respectively, in both
cases after we re-discovered the formula in Theorem~\ref{thm:main}.
We also thank David Rohrlich for his comments and encouragement,
Jean-Pierre Serre for clarifying the history of the subject, and an
anonymous referee for several valuable suggestions.

\section{Galois groups and representations}\label{s:Galois}
Throughout the paper, $F$ will be a non-archimedean local field with
residue field of characteristic $p$ and cardinality $q$.  In other
words, $F$ is a finite extension of either $\Qp$ or $\Fp((t))$.  Let
$\Fbar$ be a separable closure of $F$, $G_F$ the Galois group
$\gal(\Fbar/F)$, and $\Phi$ a geometric Frobenius element, i.e., an
element of $G_F$ which induces the inverse of the $q$-power Frobenius
automorphism of the residue field of $\Fbar$.  Let $W_F$ be the Weil
group of $F$, the subgroup of $G_F$ inducing integral powers of the
$q$-power Frobenius on the residue field, and let $I_F$ be the inertia
subgroup of $G_F$, the subgroup acting as the identity on the residue
field.  We give $W_F$ the topology characterized by the requirement
that $I_F$ be an open subgroup. 

Let $\ell$ be a prime number distinct from $p$ and let $E$ be a finite
extension of $\Ql$, the $\ell$-adic numbers.  An $\ell$-adic
representation with values in $E$ is a continuous homomorphism
$$\rho_\ell:W_F\to\GL_n(E)$$
where $\GL_n(E)$ is given the topology induced by the metric
($\ell$-adic) topology on $E$.

We are concerned with (the exponent of) the Artin conductor of
$\rho_\ell$, which we denote $a(\rho_\ell)$ and call simply the
conductor.  It will be defined in Section~\ref{s:SerreConductor}
below.

\section{Ramification groups}\label{s:ram}
In this section, we review the lower and upper ramification
filtrations on Galois groups.  See \cite[Ch.~IV]{SerreLF} for more
details.

Let $K$ be a finite Galois extension of $F$ with group $G=\gal(K/F)$.
We write $\O_K$ for the ring of integers of $K$, $\pi_K$ for a
generator of the maximal ideal of $\O_K$, and $v_K$ for the valuation
of $K$ with $v_K(\pi_K)=1$.

The ramification filtration on $G$ in the lower numbering is defined
by the requirement that
$$\sigma\in G_i\quad\Longleftrightarrow\quad v_K(\sigma(x)-x)\ge
i+1\quad\forall x\in \O_K$$
for $i$ an integer $\ge-1$.  Clearly $G_{-1}=G$, $G_0$ is the inertia
subgroup of $G$, and $G_i=0$ for all sufficiently large $i$.  By
convention, if $r\ge-1$ is a real number, we set $G_r=G_i$ where $i$
is the smallest integer $\ge r$.

Let $\varphi:[-1,\infty)\to[-1,\infty)$ be the continuous, piecewise
linear function with $\varphi(-1)=-1$, slope $1$ on $[-1,0)$, and
slope $1/[G_0:G_i]$ on $(i-1,i)$.  Let $\psi=\varphi^{-1}$,
the inverse function.  The upper numbering of the ramification
filtration on $G$ is given by
$$G^s=G_{\psi(s)}\qquad\text{and}\qquad G^{\varphi(r)}=G_r.$$
Note that the breaks in the upper numbering (i.e., the values $s$ so
that $G^{s+\epsilon}\neq G^s$ for all $\epsilon>0$) are in general
rational numbers, not necessarily integers.

The upper numbering is adapted to quotients in the following sense:
if $L/F$ is a Galois extension with $L\subset K$ and $H=\gal(L/F)$, so
that $H$ is a quotient of $G$, then the upper numbering satisfies
$$H^s=\im(G^s\to H).$$

This property allows us to define a ramification filtration on
$G_F=\gal(\Fbar/F)$ by declaring that
$$G_F^s=\left\{\sigma\in G_F\,\left|\, 
\sigma_{|_K}\in\gal(K/F)^s\ \forall K\right.\right\}$$
where $K$ runs through all finite Galois extensions of $F$.
Clearly we have $G_F^{-1}=G_F$ and $G_F^0=I_F$.  

We define
$$G_F^{>0}=\cup_{\epsilon>0}G_F^\epsilon$$ 
where the union is over all positive real numbers $\epsilon$.  We also
write $P_F$ for $G_F^{>0}$ and call this the wild inertia group of
$F$.  It is known to be a pro-$p$ group, and the quotient $I_F/P_F$ is
isomorphic as a profinite group to $\prod_{\ell\neq p}\Z_\ell$.

\section{Artin Conductor}\label{s:ArtinConductor}
In this section we review the definition of the Artin conductor of a
representation of $G=\gal(K/F)$ where $K/F$ is a finite Galois
extension.  The standard reference for this material is
\cite[Ch.~VI]{SerreLF}.

Let $\rho:G\to\GL_n(E)$ be a representation where $E$ is a field of
characteristic zero.  We write $V$ for the space where $\rho$ acts,
namely $E^n$, and for a subgroup $H$ of $G$ we write $V^H$ for the
invariants under $H$:
$$V^H=\left\{v\in V\left|\rho(h)(v)=v\ \forall h\in H\right.\right\}.$$

Recall the ramification subgroups $G_i$ of the previous section.  For
a subspace $W$ of $V$, we write $\codim W$ for the codimension of $W$
in $V$, i.e., $\dim V-\dim W$.  Following Artin \cite{Artin31b}, we
define the \emph{Artin conductor} of $\rho$ as
\begin{equation}\label{eq:conductor}
a(\rho):=\sum_{i=0}^\infty\frac{\codim V^{G_i}}{[G_0:G_i]}.
\end{equation}
Note that this is in fact a finite sum and that it depends only on the
restriction of $\rho$ to $G_0$, the inertia subgroup of $G$.  It true
but not at all obvious that $a(\rho)$ is an integer; see \cite[Ch.~VI,
\S2, Thm.~$1'$]{SerreLF}.

Because the definition of $a(\rho)$ depends only on $\rho$ restricted
to inertia, we may extend it to representations $\rho$ which are only
assumed to have finite image after restriction to inertia.

We give two alternate expressions for $a(\rho)$ which will be useful
in what follows.  First, we have
$$a(\rho)=\int_{-1}^\infty\frac{\codim V^{G_r}}{[G_0:G_r]}\,dr$$
because the integrand is constant on intervals $(i-1,i)$ and the
corresponding Riemann sum for the integral is exactly the sum defining
$a(\rho)$.  Second,
\begin{equation*}
a(\rho)=\int_{-1}^\infty{\codim V^{G^s}}\,ds.
\end{equation*}
This follows from the previous expression and the definition of the
function $\varphi$ relating the upper and lower numberings.  Indeed,
if $s=\varphi(r)$, then $ds=\varphi'\,dr$ and
$\varphi'(r)=1/[G_0:G_i]$ for $r\in(i-1,i)$.

This last formula for $a(\rho)$ turns out to be useful as it
generalizes without change to $\ell$-adic representations.  It also
makes evident the fact that if $\rho$ factors through $H=\gal(L/F)$
for some subextension $L\subset K$, then the conductor of $\rho$ as a
representation of $G$ is the same as the its conductor as a
representation of $H$.

For use later we note that the first term in the sum for $a(\rho)$ and
the first part of the integrals for it are all equal:
\begin{equation*}
\epsilon(\rho):=\int_{-1}^0{\codim V^{G^s}}\,ds=
\int_{-1}^0\frac{\codim V^{G_r}}{[G_0:G_r]}\,dr=
\codim V^{G_0}=
\codim V^{I_F}.
\end{equation*}

It is convenient to break $a(\rho)$ into two parts, $\epsilon(\rho)$
as above, and 
\begin{equation*}
\delta(\rho):=\int_0^\infty{\codim V^{G^s}}\,ds=
\int_{0}^\infty\frac{\codim V^{G_r}}{[G_0:G_r]}\,dr=
\sum_{i=1}^\infty\frac{\codim V^{G_i}}{[G_0:G_i]},
\end{equation*}
so that $a(\rho)=\epsilon(\rho)+\delta(\rho)$.   One calls
$\epsilon(\rho)$ the \emph{tame conductor} of $\rho$ and
$\delta(\rho)$ the \emph{wild conductor} or \emph{Swan conductor} of
$\rho$.

The Artin and Swan conductors may also be realized as the inner
products of the character of $\rho$ with those of certain
representations, known as the \emph{Artin representation} and the
\emph{Swan representation} respectively.  The existence of
representations with this property is essentially equivalent to the
fact that the Artin and Swan conductors are integers, see \cite[VI,
\S2]{SerreLF}.

\section{$\ell$-adic representations}\label{s:l-adic}
Recall that an $\ell$-adic representation is a continuous
homomorphism 
$$\rho_\ell:W_F\to\GL_n(E)$$
where $E$ is a finite extension of $\Ql$ and $\GL_n(E)$ is given the
$\ell$-adic topology.  (When $E\neq\Ql$, some authors refer to
$\rho_\ell$ as a $\lambda$-adic representation.)  A primary source of
such representations is $\ell$-adic cohomology.  More precisely, if
$X$ is a variety over $F$, then the $\ell$-adic \'etale cohomology
groups $H^i(X\times\Fbar,\Ql)$ (and variants) are equipped with
continuous actions of $G_F$ and we may restrict to $W_F$ to obtain
$\ell$-adic representations as defined above.

Let $\O_E$ be the ring of integers of $E$ and let $\m\subset\O_E$ be
the maximal ideal. It is well known (see, e.g.,
\cite[p.~104]{SerreLALG}) that every compact subgroup of $\GL_n(E)$ is
conjugate to a subgroup of $\GL_n(\O_E)$.  Since $I_F$ is a closed
subgroup of $G_F$, it is compact.  Thus replacing $\rho_\ell$ by a
conjugate representation if necessary, we may assume that the image of
$I_F$ under $\rho_\ell$ is contained in $\GL_n(\O_E)$.

Now the $\GL_n(\O_E)$ has a finite index subgroup
which is a pro-$\ell$ group, namely, the kernel of reduction
$\GL_n(\O_E)\to\GL_n(\O_E/\m)$.   Since $P_F\subset I_F$ is a pro-$p$
group and $\ell\neq p$, it follows that the image of $P_F$ under
$\rho_\ell$ is finite.

On the other hand, it is not in general the case that the image of all
of $I_F$ under $\rho_\ell$ is finite.  For example, if $\rho_\ell$ is
the representation of $G_F$ on the Tate module of an elliptic curve
over $F$ with split multiplicative reduction, it is known that the
restriction of $\rho_\ell$ to $P_F$ is trivial, and in a suitable
basis the restriction of
$\rho_\ell$ to $I_F$ has the form
$$\left(\begin{matrix}1&\tau\\0&1\end{matrix}\right)$$
where $\tau:I_F\to\Z_\ell$ is the projection $I_F\to
I_F/P_F\cong\prod_{r\neq q}\Z_r\to\Z_\ell$.  See
\cite[\S15]{Rohrlich94} for more details.

Since the restriction of $\rho_F$ to inertia need not have finite
image, the definition of Artin does not apply directly because the
denominators $[G_0:G_i]$ in \eqref{eq:conductor} need not be finite,
so further discussion is needed.

The crucial observation that makes it possible to define a conductor
for $\rho_\ell$ is the following, which was conjectured by Serre and
Tate and proven by Grothendieck.

\begin{kf}
\label{thm:finite-index}
If $\rho_\ell:W_F\to\GL_n(E)$ is a continuous representation, then
there is a subgroup $I'\subset I_F$ of finite index such that
$\rho_\ell(g)$ is unipotent for every $g\in I'$.
\end{kf}

The proof is given in \cite[Appendix]{SerreTate68}.  It applies more
generally to discrete valuation fields $F$ whose residue field has the
property that no finite extension contains all roots of unity of
$\ell$-power order.  Note that we may assume that $I'$ is also normal
in $W_F$.

\section{Semi-simplification}\label{s:SerreConductor}
Fix an $\ell$-adic representation $\rho_\ell:W_F\to\GL_n(E)$.  We
write $V_\ell$ for $E^n$ with its action of $W_F$ via $\rho_\ell$.

Let $V_{ss}$ be the semi-simplification of $V_\ell$, defined as the
direct sum of the Jordan-H\"older factors of $V_\ell$ as a $W_F$
module, and let $\rho_{ss}:W_F\to\GL_n(E)$ be the corresponding
homomorphism.  

It follows from the key fact of the previous section that $\rho_{ss}$
is trivial on a finite index subgroup of $I_F$.  Indeed, by Clifford's
theorem \cite[Thm.~1]{Clifford37}, $\rho_{ss}$ is semi-simple when
restricted to any normal subgroup of $W_F$, and by Kolchin's theorem
\cite[V.3*, p.~35]{SerreLALG}, a semi-simple and unipotent
representation is trivial.  In particular, $\rho_{ss}(I_F)$ is a
finite group.  Therefore, the wild conductor $\delta(\rho_{ss})$ and
the Artin conductor $a(\rho_{ss})$ are well defined.

Following Serre \cite{Serre70}, we define
\begin{equation*}
a(\rho_\ell):=\epsilon(\rho_\ell)+\delta(\rho_{ss})
\end{equation*}
where as usual $\epsilon(\rho_\ell)=\dim V_\ell-\dim V_\ell^{I_F}$. 

Note that if $\rho_\ell$ restricted to inertia has finite image, then
$\rho_\ell$ and $\rho_{ss}$ have the same restriction to inertia, so
this definition agrees with that in Section~\ref{s:ArtinConductor}
when they both apply.

Note also that 
$$\epsilon(\rho_\ell)-\epsilon(\rho_{ss})=
\dim V_{ss}^{I_F}-\dim V_\ell^{I_F}$$
so we have
$$a(\rho_\ell)=\dim V_{ss}^{I_F}-\dim V_\ell^{I_F}+a(\rho_{ss}).$$
This last formula appears in \cite[4.2.4]{Tate79}, but is missing the
exponent $I_F$ on $V_\ell$.

The wild conductor $\delta(\rho_{ss})$ has two other useful
descriptions.  We noted above that $\rho_\ell|_{P_F}$ (the restriction
to wild inertia) has finite image, so is already semi-simple, and
$\rho_{ss}|_{P_F}$ is also semi-simple by the same reasoning (or by
Clifford's theorem).  Since $\rho_\ell|_{P_F}$ and $\rho_{ss}|_{P_F}$
have the same character, they are isomorphic and have the same wild
conductor.  For another description, choose a basis of $V_\ell$ so
that the image of $\rho_\ell$ lies in $\GL_n(\O_E)$ and let
$\overline{\rho}_\ell$ be the reduction modulo $\m_E$.  Then
$\overline{\rho}_\ell$ restricted to $P_F$ is isomorphic to
$\rho_\ell$ restricted to $P_F$ and also gives the same wild
conductor.  Summarizing:
\begin{equation*}
\delta(\rho_{ss})=\delta(\rho_\ell|_{P_F})=\delta(\overline{\rho}_\ell).
\end{equation*}

\section{Weil-Deligne representations}\label{s:WD}
In this section we review (with the minimum of details) the notion of
a Weil-Deligne representation.  See \cite[\S8]{Deligne73} or
\cite[4.1]{Tate79} for more details, and \cite{Rohrlich94} for a
motivated introduction aimed at arithmetic geometers.

We write $||\cdot||$ for the homomorphism $W_F\to\Q$ which sends
$\Phi$ to $q^{-1}$ and which is trivial on $I_F$.

Let $V$ be a vector space over a field of characteristic 0.  We define
a Weil-Deligne representation of $W_F$ on $V$ as a pair $(\rho,N)$
where $\rho:W_F\to\GL(V)$ is a homomorphism continuous with respect to
the discrete topology on $V$ and $N:V\to V$ is an endomorphism
satisfying
$$\rho(w)N\rho(w)^{-1}=||w||\,N.$$
Continuity of $\rho$ implies that is has finite image restricted to
inertia, and the displayed formula implies that $N$ is nilpotent
(because its eigenvalues are stable under multiplication by $q$).

Because $\rho$ has finite image when restricted to $I_F$, its
conductor is defined by the formulas of
Section~\ref{s:ArtinConductor}.  We define the Artin conductor of a
Weil-Deligne representation $(\rho,N)$ as
$$a(\rho,N):=a(\rho)+\dim V^{I_F}-\dim V^{I_F}_N.$$
Here $V_N$ is the kernel of $N$ on $V$, so that
$$V^{I_F}_N=\left\{v\in V\left| 
   N(v)=0, \rho(w)(v)=v\ \forall w\in I_F\right.\right\}.$$

\section{$\ell$-adic representations and Weil-Deligne
  representations}\label{s:l-WD} 
Fix an $\ell$-adic representation $\rho_\ell$.  The key fact stated at
the end of Section~\ref{s:l-adic} leads to a description of the
behavior of $\rho_\ell$ restricted to $I_F$ in terms of Weil-Deligne
representations.
 
First, note that because $P_F$ is a pro-$p$ group and
$I_F/P_F\cong\prod_{\ell\neq p}\Z_l$, there is a non-zero homomorphism
$t_\ell:I_F\to\Ql$ which is unique up to a scalar.  It satisfies
$t_\ell(w\sigma w^{-1})=||w||\,t_\ell(\sigma)$ for all $w\in W_F$.

The key fact implies that there is a unique nilpotent linear
transformation $N:E^n\to E^n$ such that for all $\sigma$ in some
finite index subgroup of $I_F$
$$\rho_\ell(\sigma)=\exp(t_\ell(\sigma)N)$$
as automorphisms of $E^n$.  Here $\exp$ is defined by the usual series
$1+x+x^2/2!+\cdots$ and $\exp(t_l(\sigma)N)$ is in fact a finite sum
because $N$ is nilpotent.

It follows from this (see \cite[\S8]{Deligne73}) that there exists a
unique Weil-Deligne representation $(\rho,N)$ on $V=E^n$ such that for
all $m\in\Z$ and all $\sigma\in I_F$
\begin{equation}\label{eq:WD}
\rho_\ell(\Phi^m\sigma)
=\rho(\Phi^m\sigma)\exp\left(t_\ell(\sigma)N\right).
\end{equation}

Conversely, given a Weil-Deligne representation $(\rho,N)$ on $V$, the
displayed formula defines an $\ell$-adic representation.  This
correspondence gives a bijection on isomorphism classes.  (The
correspondence $\rho_\ell\leftrightarrow(\rho,N)$ depends on the
choices of $t_\ell$ and $\Phi$, but after passing to isomorphism
classes it is independent of these choices, see \cite{Deligne73}.)

The point of introducing Weil-Deligne representations is that their
definition uses only the discrete topology on $V$, so is convenient for
shifting between different ground fields (such as $\Ql$ for varying
$\ell$ and $\C$).

Following Deligne \cite{Deligne73}, we define
$$a_D(\rho_\ell):=a(\rho,N)=a(\rho)+\dim V^{I_F}-\dim V^{I_F}_N.$$
We note that $\rho_\ell(I_F)$ is finite if and only if the
corresponding $N=0$, and in this case the definition above reduces to
that of Section~\ref{s:ArtinConductor}.

Since $\epsilon(\rho)=\dim V-\dim V^{I_F}$, we also have
$$a_D(\rho_\ell)=\dim V-\dim V^{I_F}_N+\delta(\rho).$$
We note also that $t_\ell$ is trivial on the wild inertia group
$P_F=G_F^{>0}$, so $\rho_\ell$ and $\rho$ are equal on $P_F$.  It follows
that $\delta(\rho)=\delta(\rho_\ell|_{P_F})=\delta(\rho_{ss})$.  On
the other hand, equation~\eqref{eq:WD} (with $m=0$) implies that
$V_N^{I_F}=V_\ell^{I_F}$.  Therefore
$$a_D(\rho_\ell)=\dim V-\dim V^{I_F}_\ell+\delta(\rho_\ell)=a(\rho_\ell),$$
in other words, the definitions of Deligne and Serre agree.

\section{Another formula for $a(\rho_\ell)$}\label{s:main}
We can now state the main result of this note.  The left hand side of
\eqref{eq:int} appears as a definition in \cite[Def.~3.1.27]{WieseGR}.
This reference seems to include everything needed to prove that
Definition~3.1.27 agrees with the definitions of Serre and Deligne,
but the proof is not given there.

\begin{thm}\label{thm:main}
  Let $\rho_\ell$ be an $\ell$-adic representation of $W_F$ on
  $V_\ell$ with corresponding Weil-Deligne representation $(\rho,N)$
  on $V$ and semisimplification $\rho_{ss}$ on $V_{ss}$.  Let
  $a(\rho_\ell)$ be Artin conductor of $\rho_\ell$, defined as in
  Section~\ref{s:SerreConductor}, and let $a_D(\rho_\ell)$ be defined
  as in Section~\ref{s:l-WD}.  Then
\begin{equation}\label{eq:int}
\int_{-1}^\infty \codim V_\ell^{G^s}\,ds=a(\rho_\ell)=a_D(\rho_\ell).
\end{equation}
\end{thm}

\begin{proof}
  We saw at the end of Section~\ref{s:l-WD} that
  $a(\rho_\ell)=a_D(\rho_\ell)$, so we need only check that the
  integral is equal to $a(\rho_\ell)$.  For $-1<s<0$, $G^s=G_s=I_F$,
  so
$$\int_{-1}^0 \codim V_\ell^{G^s}\,ds=\codim V_\ell^{I_F}
=\epsilon(\rho_\ell).$$
On the other hand, for $s>0$, $G^s\subset P_F$ and $\rho_\ell$
restricted to $G^s$ is isomorphic to $\rho_{ss}$ restricted to $G^s$.
Therefore
$$\int_0^\infty \codim V_\ell^{G^s}\,ds=
\int_0^\infty \codim V_{ss}^{G^s}\,ds=
\delta(\rho_{ss}).$$
It follows that 
$$\int_{-1}^\infty \codim V_\ell^{G^s}\,ds=
\epsilon(\rho_\ell)+\delta(\rho_{ss})=a(\rho_\ell)$$ 
as desired.
\end{proof}

\section{An application to twisting}\label{s:twist}
We give an easy application of the theorem which is the motivation for
this work.

Let $\rho_\ell:W_F\to\GL_n(E)$ be an $\ell$-adic representation and
let $\chi:W_F\to E^\times$ be a character.  We say ``$\chi$ is more
deeply ramified than $\rho_\ell$'' if there exists a non-negative real
number $s$ such that $\rho_\ell(G_F^s)=\{id\}$ and
$\chi(G_F^s)\neq\{id\}$.  In other words, $\chi$ is non-trivial
further into the ramification filtration than $\rho_\ell$ is.  Let
$m$ be the supremum of the set of $s$ such that $\chi$ is non-trivial on
$G_F^{s}$.  It follows from Section~\ref{s:ArtinConductor} that
$a(\chi)=m+1$.

\begin{prop}\label{lemma:ramification}
  If $\chi$ is more deeply ramified than $\rho_\ell$, then
$$a(\rho_\ell\tensor\chi)=\deg(\rho_\ell)a(\chi).$$
\end{prop}

\begin{proof}
  Let $V_\ell$ be the space where $W_F$ acts via $\rho_\ell$ and let
  $V_{\ell,\chi}$ be the same space where $W_F$ acts via
  $\rho_\ell\tensor\chi$.  By the theorem we have
$$a(\rho_\ell\tensor\chi)=\int_{-1}^\infty \codim V_{\ell,\chi}^{G_F^s}\,ds.$$
If $s\le m$ then $V_{\ell,\chi}^{G_F^s}\subset
V_{\ell,\chi}^{G_F^{m}}$ and the latter is zero because
$\rho_\ell(G_F^{m})=\{id\}$ and $\chi(G_F^{m})\neq\{id\}$.  Thus
in this range the integrand is $\dim V_\ell=\deg(\rho_\ell)$.  On the
other hand, if $s>m$, then $\rho_\ell\tensor\chi(G_F^{s})=\{id\}$
and the integrand is zero.  Thus
 $$\int_{-1}^\infty \codim V_{\ell,\chi}^{G_F^s}\,ds=\deg(\rho_\ell)(m+1)
=\deg(\rho_\ell)a(\chi)$$
as desired.
\end{proof}

A particularly useful case of the proposition occurs when $\rho_\ell$
is tamely ramified and $\chi$ is wildly ramified, e.g., when $\chi$ is
an Artin-Schreier character.

A variant of the proposition where $\chi$ and $\rho_\ell$ are both
assumed to be irreducible, but $\chi$ may be of dimension $>1$, is
stated as Lemma~9.2(3) of \cite{Dokchitsers13}

We end with another application in the same spirit, namely a very
simple solution to Exercise~2 in \cite[p.~103]{SerreLF}.

\begin{prop}
  Suppose that $\rho_\ell$ is an irreducible $\ell$-adic
  representation of $W_F$ on $V_\ell$, and let $m$ be the supremum of
  the set of numbers $s$ such that $\rho_\ell(G^s)\neq0$.  Then
$$a(\rho_\ell)=(\dim V_\ell)(m+1).$$
\end{prop}

\begin{proof}
  Since $G^s$ is a normal subgroup of $G_F$, the subspace of
  invariants $V_\ell^{G^s}$ is preserved by $G_F$.  Since $V_\ell$ is
  irreducible, we have that $\codim V_\ell^{G^s}$ is 0 if $s>m$ and
  $\dim V_\ell$ if $s<m$.  Therefore
$$a(\rho_\ell)=\int_{-1}^\infty \codim V_\ell^{G_F^s}\,ds=
(\dim V_\ell)(m+1)$$
as desired.
\end{proof}

\bibliography{database}{}
\bibliographystyle{alpha}

\end{document}